\documentclass[a4paper]{article}

\usepackage{a4} 
\usepackage{amsfonts}
\usepackage{amssymb}
\usepackage{amsmath}
\usepackage{amscd}
\usepackage{amsthm}
\usepackage{epsfig}
\usepackage{color}
\usepackage{bbm}
\usepackage{graphicx}
\usepackage{hyperref}
\usepackage{enumerate}
\usepackage{multirow}

\newcommand{\R}{\mathbb{R}}

\newcommand{\bprf}{\begin{proof}[Proof]}
\newcommand{\eprf}{\end{proof}}

\newtheorem{thm}{Theorem}

\newtheorem{lem}[thm]{Lemma}

\theoremstyle{definition} 
\newtheorem{rem}[thm]{Remark}

\begin{document}
\title{On the number of colored Birch and \\Tverberg partitions}
\author{{\Large Stephan Hell}
} 

\maketitle
\begin{abstract}
In 2009, Blagojevi\'c, Matschke \& Ziegler established the first tight colored Tverberg
theorem, but no lower bounds for the number of colored Tverberg partitions. We develop a
colored version of our previous results (2008), and we extend our results from the
uncolored version: Evenness and non-trivial lower bounds for the number of colored Tverberg partitions.
This follows from similar results on the number of colored Birch partitions.
\end{abstract}
\begin{section}{Introduction}\label{sec-intro}
In 1966, Tverberg~\cite{tverberg66:_theor} showed that any $(d+1)(r-1)+1$ points in $d$-dimensional space $\R^d$
can be partitioned into $r$ blocks whose convex hulls have a non-empty intersection. This result is
known as {\it Tverberg's theorem}, and it has several proofs, and many generalizations, 
see Matou\v{s}ek~\cite[Sect.~6.5]{matou03:_using_borsuk_ulam} for details.

The first {\it colored Tverberg theorem} is due to B\'ar\'any \& Larman~\cite{barany92}; see Ziegler~\cite{z11:_col-tvp_ams}
for a recent account of the story. In 2009, Blagojevi\'c, Matschke \&
Ziegler~\cite{bmz10:_col-tp} established an {\it optimal} colored Tverberg theorem. Since then, their results have been reproved
by themselves~\cite{bmz11:_col-tvp},
Matou\v{s}ek, Tancer \& Wagner~\cite{mtw11:_col-geom-ctvp}, 
and Vre\'{c}ica \& \v{Z}ivaljevi\'c~\cite{vrecica-zival-11:chessboard}. 
\begin{thm}[{\cite[Thm~2.2]{bmz10:_col-tp}}]\label{thm-ctv-bmz11}Let $d \geq 1$, $r \geq 2$ prime, $N:=(d+1)(r-1)$, and 
$f:\Delta_N\rightarrow \R^d$ continuous, where the $N+1$ vertices of $\Delta_N$ have $d+2$ different colors, 
and the color classes satisfy $|C_0|=|C_1|=\ldots=|C_d|=r-1$ and $|C_{d+1}|=1$. Then the simplex $\Delta_N$ 
has $r$ disjoint rainbow faces $F_1,F_2,\ldots,F_r$ whose images under $f$ have a non--empty intersection:
\[ \bigcap_{i=1}^r f(F_i) \not = \emptyset.\]
\end{thm}
Here {\it rainbow} means that every color occurs at most once. In the following, we focus on the case when $f$ is an
{\it affine} map. In this case, one can think of the set $f({\rm vert}(\Delta_N))\subset\R^d$ as $N+1$ colored points
satisfying the above color condition which can be partitioned into $r$ rainbow partition blocks $B_1,B_2,\ldots,B_r$, where
$B_i=f({\rm vert}(F_i))$ for all $i$, such that 
their convex hulls intersect:
\[ \bigcap_{i=1}^r \text{conv} (B_i) \not = \emptyset.\]

Both Tverberg's theorem, and Theorem~\ref{thm-ctv-bmz11} settle the existence of one (!) partition. In the uncolored
case, Sierksma conjectured that there are at least $((r-1)!)^d$ partitions based on a particular point configuration;
see~\cite{matou03:_using_borsuk_ulam}. This conjecture is open for $d\geq 2$.
{\it Lower bounds} for the number of Tverberg partitions have first been obtained by Vu\'ci\'c \& 
\v{Z}ivaljevi\'c~\cite{vz93:_notes_sierk} when $r$ is prime, and by the author~\cite{hell07:_tverb} when $r$ is a prime power. 
Then the first lower bound was shown that holds for $r$ arbitrary in~\cite{hell08:_birch}. Up to now, no non-trivial
lower bounds have been known in the colored case, not even a good conjecture. This is what
we provide here: 
{\it Lower bounds for the number of colored Tverberg partitions that hold for arbitrary $r$.}
We extend our approach from the uncolored case in~\cite{hell08:_birch}: We study colored Birch partitions
in Theorem~\ref{thm-num-cbp} which yields the first non-trivial lower bounds in Theorem~\ref{thm-num-ctv}. 
In Section~\ref{sec-rem}, we discuss minimal point configurations.\\

{\it Observation.}
Assuming that the $(d+1)(r-1)+1$ points are in general position, the partition blocks consist of at most $d+1$ points. 
One possible solution is a single point that lies in the convex hulls of $r-1$ many $(d+1)$-element sets. 
The other extreme case would be 
$d$ partition blocks of exactly (!) $d$ points each, intersecting in a single point, plus $r-d$ many $(d+1)$-element sets
that all contain the intersection point, where $d\leq r$. In all cases, we have at least $r-d$ many $(d+1)$-element 
sets $B_1, B_2,\ldots , B_{r-d}$
which (i) contain a common point in their convex hulls, and which (ii) are rainbow sets in the following way: 
each of them contains each of the colors $0,1,\ldots, d$ exactly once, that is, $|B_i\cap C_j|=1$ for all $0\leq j\leq d$ and 
all $1\leq i \leq r-d-1$. 

This observation leads to the concept of 
colored Birch partitions. For this, let $p\in\R^d$ be a point, and $k\geq 1$ a natural number. Given a set X of 
$k(d+1)$ colored points in $\R^d$ of $d+1$ different colors such that each color class 
$C_0,C_1,\ldots, C_d$ contains exactly $k$ points, we call a partition $B_1,B_2,\ldots,B_k$ a {\it colored 
Birch partition of $X$ to the point $p$} if each block $B_i$ contains 
exactly $d+1$ points, uses every color exactly once, and contains $p$ in its convex hull. Let $\text{cBP}_k(X)$ be
the number of all unordered colored Birch partitions of $X$ to $p$. Here {\it unordered} means 
that two partitions are regarded as the same if one can be obtained from the other by a permutation of the $k$ partition
blocks. The partitions in the
previous paragraph are examples of colored Birch partitions to the single point resp.~the intersection point.
Placing $p$ outside the convex hull of $X$ one gets $\text{cBP}_k(X)=0$. Figure~\ref{fig-bp} shows an 
example for $d=2$, and $k=4$ with $\text{cBP}_k(X)=2$. By definition: $\text{cBP}_k(X)\leq\text{BP}_k(X)$, where $\text{BP}_k(X)$
is the number of uncolored Birch partitions, see~\cite{hell08:_birch} for more information.\\

Let us formulate our main results. For this, a set of points is in {\it general position} of no $k+2$ 
points are on a common $k$-dimensional affine subspace.
\begin{thm}\label{thm-num-cbp}For $d\geq 1$, let $p\in\R^d$ be a point, and $k\geq 1$ a natural number. For any set X of 
$k(d+1)$ colored points in $\R^d$ of $d+1$ different colors such that each color class $C_0,C_1,\ldots, C_d$ 
contains exactly $k$ points, and $X\cup\{p\}$ in general position, 
the number of colored Birch partitions $\text{cBP}_k(X)$ has the following
four properties:
\begin{enumerate}[\rm (i)]
\item\label{it-cb-even}$\text{cBP}_k(X)$ is even for $k\geq d+2$.
\item\label{it-cb-lower-d1}$\text{cBP}_k(X)>0\,\,\Longrightarrow\,\,\text{cBP}_k(X)\geq \lceil \frac{k}{2}\rceil !
\cdot \lfloor \frac{k}{2}\rfloor !$
for $d=1$.
\item\label{it-cb-lower-plane}$\text{cBP}_k(X)>0\,\,\Longrightarrow\,\,\text{cBP}_k(X)\geq 8\cdot 3^{k-6}$
for $d=2$ and $k\geq 6$.
\item\label{it-cb-lower}$\text{cBP}_k(X)>0\,\,\Longrightarrow\,\,\text{cBP}_k(X)\geq 2^{k-d-1}$ for $d \geq 2$ and $k\geq d+2$. 
\end{enumerate}
\end{thm}
The condition $k\geq d+2$ in (\ref{it-cb-even}) is necessary for $d=2,3,4$ as there are 
counter-examples for $k=d+1$. Computer experiments for dimensions 2, and 3 show that the 
lower bounds (\ref{it-cb-lower-plane}) and (\ref{it-cb-lower}) are tight: 
For $4\leq k \leq 9$ in dimension 2, and for $5\leq k \leq 8$ in dimension~3.

In the following, we construct a planar set $X$ such that ${cBP}_3(X)$ is odd.
In the planar setting, a point configuration can be represented as a colored word of length $3k$ on the 
alphabet $\{+,-\}$: Choose a line through $p$. This line hits at most one point
from $X$, and it divides the plane into two half-spaces. Choose one of
the two half-spaces. Then sweep the line through $p$ over the
chosen half-space counter-clockwise.  The ray hits all points exactly
once, and the sweeping leads to a linear order on the points in $X$.
This determines a colored word of length $3k$ on the alphabet $\{+,-\}$
in the following way: Write for every point of $X$ the letter $+$ when
the line hits a point in the chosen half-space, and $-$ in the
other case. While writing the letters, keep for each letter track of its color.

Every possibility of partitioning a colored word of length $3k$ into $k$ colored subwords of
the form $+-+$, or $-+-$ corresponds one-to-one to a colored Birch partition of $X$.
One can check that the alternating word $+-+-+-+-+$ of length $9$ with a cyclic coloring $0,1,2,0,1,2,0,1,2$ 
corresponds to a colored point configuration with $\text{cBP}_3(X)=3$ being odd. Namely, 
one partition is $\{0,1,2\}, \{3,4,5\}, \{6,7,8\}$, where the letters are numbered from left to right. The other
two are\linebreak $\{0,1,8\},\{2,3,4\},\{5,6,7\}$, and 
$\{0,7,8\},\{1,2,3\},\{4,5,6\}$.\\

Theorem~\ref{thm-num-cbp} implies analogous properties for the number of colored Tverberg partitions.
\begin{thm}\label{thm-num-ctv}Let $d \geq 1$, $r \geq 2$ prime, $N:=(d+1)(r-1)$, and 
$f:\Delta_N\rightarrow \R^d$ affine, where the $N+1$ vertices of $\Delta_N$ 
have $d+2$ different colors, 
and the color classes satisfy: $|C_0|=|C_1|=\ldots=|C_d|=r-1$ and $|C_{d+1}|=1$. Then the number of unordered
colored Tverberg partitions $T(f)$ satisfies the following
four properties:
\begin{enumerate}[\rm (i)]
\item\label{it-ctv-even}$T(f)$ is even for $r\geq 2d+2$.
\item\label{it-ctv-lower-d1}$T(f)\geq \lceil \frac{r-1}{2}\rceil !
\cdot \lfloor \frac{r-1}{2}\rfloor !$
for $d=1$.
\item\label{it-ctv-lower-plane}$T(f)\geq 8\cdot 3^{r-8}$, for $d=2$ and $r\geq 8$.
\item\label{it-ctv-lower}$T(f) \geq 2^{r-2d-1}$, for $d \geq 2$ and $r\geq 2d+2$. 
\end{enumerate}
\end{thm}
It is an easy exercise to show that the lower bound for $d=1$ is optimal. In general, the lower
bounds might not be optimal as we assumed that there is (1) only one colored Tverberg point being (2) the
intersection point of $d$ partition blocks of exactly $d$ points each.
We have not found any (uncolored) example having both properties at the same time. Assuming
that the colored Tverberg point is one of the vertices of $\Delta_N$ leads to a lower bound of $8\cdot 3^{r-7}$
resp.~$2^{r-d-2}$ for sufficiently large $r$. For further remarks, see our discussion based on computations 
in Section~\ref{sec-rem}.

In Section~\ref{sec-proof-cbp} we prove Theorem~\ref{thm-num-cbp}, 
and Theorem~\ref{thm-num-ctv} in Section~\ref{sec-proof-ctv}.
\begin{figure}
\centering
\includegraphics[width=9cm]{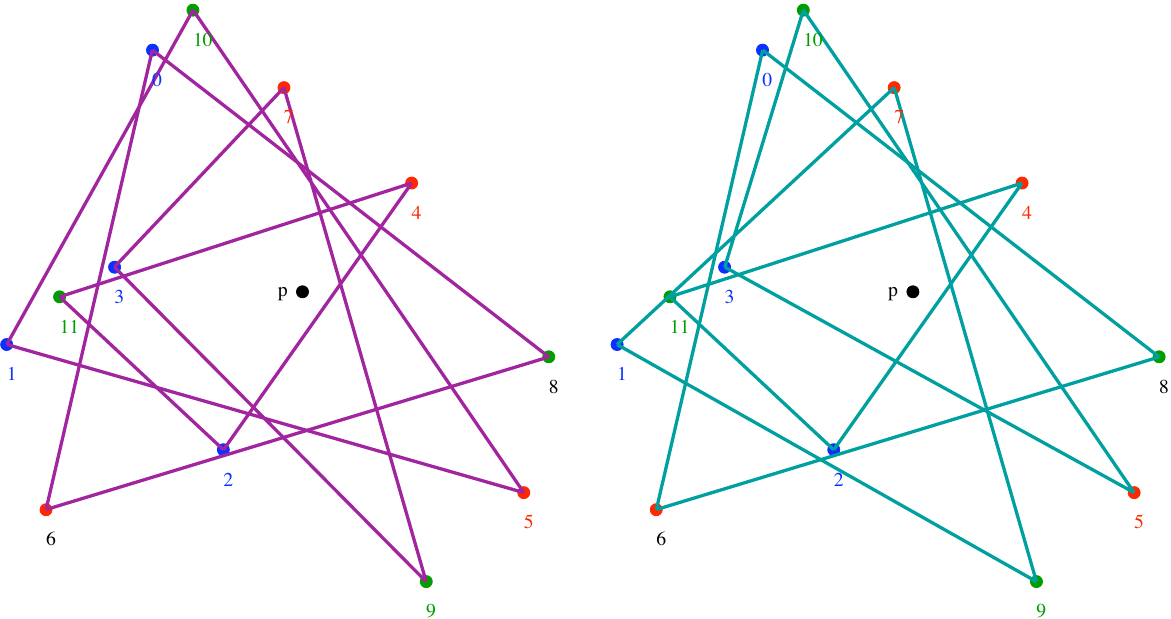}
\caption{A planar example for $k=4$ with $\text{cBP}_k(X)=2$.}
\label{fig-bp}
\end{figure}
\end{section}
\begin{section}{Proof of Theorem~\ref{thm-num-cbp}}\label{sec-proof-cbp}
Property~(\ref{it-cb-lower-d1}) is an easy exercise. We first prove Property~(\ref{it-cb-even}) inductively; here the key part is the base case $k=d+2$.
In a second step, we show that Property~(\ref{it-cb-even}) implies 
Properties~(\ref{it-cb-lower-plane}) and~(\ref{it-cb-lower}).

We will use an approach similar to the uncolored case in~\cite{hell08:_birch}: One of
our points will be moved while all the others remain fixed. During this moving process, we will keep
track of the parity for the number of colored Birch partitions. 

In the following, we assume $d\geq 2$. Let $k\geq 2$, fix $p$ to be the origin $o\in\R^d$, and assume 
without restriction that all $k(d+1)$ 
colored points of $X$ are on the unit 
sphere $S^{d-1}\subset \R^d$. If all points lie in the northern hemisphere of $S^{d-1}$, then
$\text{cBP}_{k}(X)=0$, as the origin is not in the convex hull of $X$. Below we do the following:
We move one colored 
point $q$ while fixing all others. It is sufficient to show that the parity of $\text{cBP}_{k}(X)$ does
not change during this.

Let $q$ be a point of $X$. Instead of looking at $q$, we follow its antipode $-q$ as 
for any $d$-element subset $S\subset X\setminus\{ q\}$, one has:
\begin{eqnarray*}
o\in\text{conv} (S\cup\{q\}) \;\Longleftrightarrow\; -q\in\text{cone}(S).
\end{eqnarray*}

From now on, we restrict ourselves to $d$-element subsets $S\subset X$ such that $S\cup \{q\}$ is rainbow. 
Every $d$-element 
subset $S$ defines a cone in $\R^d$, all these cones decompose the sphere $S^{d-1}\subset \R^d$ into cells.
As long as $-q$ moves inside one of these cells, $\text{cBP}_{k}(X)$ does not change. At some
point, we are forced to move $-q$ from one cell to another.  
At that point $\text{cBP}_{k}(X)$ might change. A boundary hyperplane of a cell is defined through a $(d-1)$-element subset $H\subset S$.

Our moving procedure can be chosen so that our cell decomposition is nice, 
and that $-q$ crosses a boundary hyperplane of the cell in a transversal way. 
Before looking at colored Birch partitions, let's look at the set $\cal A$ of all rainbow $d$-simplices
containing the origin.
If $-q$ crosses a hyperplane defined through a subset $H$, then $\cal A$ might change. Let
$H'=H\cup\{q\}$. For all colored simplices that do not contain $H'$ as a face, nothing changes. For the other
simplices $\Delta$ the following property switches:
\begin{eqnarray}\label{crossing}
o\in\text{conv} (\Delta) \text{ before the crossing.}\;\Longleftrightarrow \; o\not\in\text{conv}(\Delta)\text{ afterwards.} 
\end{eqnarray} 

A colored Birch partition of $X$ consists of $k$ disjoint rainbow $d$-simplices containing the origin. 
If $-q$ crosses a hyperplane defined through $H\subset X$, then some colored Birch partitions vanish, 
and new colored Birch partitions come up. In fact, all Birch partitions, that include a simplex
$\Delta$, $H'\subset\Delta$, which contains the origin before the crossing, vanish. 
The new ones include a simplex
$\Delta$, $H'\subset\Delta$, which contains the origin after the crossing, but only 
if $X\setminus\Delta$ admits a colored Birch partition into $k-1$ partition blocks.

In our proof, we need a special case of Deza et al.~\cite[Theorem~3.5]{deza06:_colour}
which we reprove to make the reader familiar with the argument used below.
\begin{lem}[{\cite[Theorem~3.5]{deza06:_colour}}]\label{lem-key}
For $d\geq 2$, and a given set X of 
$2(d+1)$ colored points in $\R^d$ of $d+1$ different colors such that each color occurs 
exactly twice, the number of colored $d$-simplices containing the origin is even.
\end{lem}
\bprf Let $X=\{0,1,2,\ldots,2d+1\}$ such that the points $2i,2i+1$ are of color $i$, for all
$0\leq i\leq d$. Without restriction, we choose $q=0$, and the boundary hyperplane of our cell spanned 
by $H=\{2,4,\ldots,2(d-1)\}$. If $-q$ crosses the hyperplane through $H$, then exactly two colored $d$-simplices 
$\{0,2,4,\ldots,2(d-1),2d\}$, and $\{0,2,4,\ldots,2(d-1),2d+1\}$ are affected as observed 
in~(\ref{crossing}). 
In any case, the parity for the number of colored $d$-simplices containing the origin does not change.
\eprf

\begin{proof}[Proof of Theorem~\ref{thm-num-cbp}, Property~(\ref{it-cb-even})] This follows -- as in the uncolored case -- via induction 
from its base case $k=d+2$. Let $k\geq d+3$, and $x$ a point of color 0. Let $B_1,B_2,\ldots,B_l$
be all rainbow $d$-simplices containing the origin, and using the point $x$. For every $i\in [l]$ the
set $X\setminus B_i$ has an even number of colored Birch partitions by assumption. 
Adding up all these even numbers leads to $\text{cBP}_k(X)$.

Let $k=d+2$, and $X$ be our set of $(d+1)(d+2)$ colored points. We will repeat the following step $d$
times, and then we will finally apply Lemma~\ref{lem-key} to complete our proof.\\

{\bf Step 1:} Let $q$ be a point of $X$, and the boundary hyperplane -- that is crossed transversally -- be
spanned by a rainbow set $H_1$. Assume without restriction that in $H_1\cup\{q\}$ the $d$ colors 
$0,1,2,\ldots,d-1,\widehat{d}$ show up, where $\widehat{d}$ means ``omit $d$". For every $s\in C_d$,
the colored $d$-simplex $H_1\cup\{q,s\}$ will change its property of containing the 
origin -- as observed in~(\ref{crossing}) -- so that some colored Birch partitions vanish, and new ones
come up. Again, new ones come up if the rest admits a colored Birch partition into $d+1$ blocks.
To prove the evenness of $\text{cBP}_{d+2}(X)$ it is sufficient to show
that 
\begin{eqnarray}\label{hilfslemma}
\text{$\text{cBP}_{d+1}(X_1)=\sum_{s\in C_{d}^1} \text{cBP}_{d+1}(X_1\setminus \{s\})$ is even.}
\end{eqnarray}
Here, the set $X_1=X\setminus H_1$ consists of $(d+1)^2+1$ points: The $d$ new color classes $C_0$ to
$C_{d-1}$ are of size $d+1$, and color class $C_d$ of size $d+2$. Therefore, the expression 
$\text{cBP}_{d+1}(X_1)$ stands for the sum over the $d+2$ 
possibilities to drop one of the $d+2$ points of color $d$ from 
$X_1$. Define the new color classes
$C_i^1$ to be $C_i$ minus the point of color $i$ in $H_1\cup\{q_1\}$, for $0\leq i\leq d$.

In Step 1, we have reduced the
partition parameter $k$ from $d+2$ to $d+1$ by $1$, and the number of points from $(d+1)(d+2)$ 
to $(d+1)^2+1$ by~$d$. In repeating this step $d$ times, we will end up with $k=d+2-d=2$, and
$(d+1)(d+2)-d^2=3d+2$ many points. Finally, the color class $C^d_0$ will be of size
$2$, and $C^d_1$ to $C^d_d$ of size $3$.\\

{\bf General Step $i$, $2\leq i\leq d$:} Assume that we have reduced our problem to 
showing that 
\[ \text{cBP}_{d+3-i} (X_{i-1}) =
\sum_{s_1\in C^{i-1}_d,s_2\in C^{i-1}_{d-1},\ldots,s_{i-1}\in C^{i-1}_{d-i+2}}
\text{cBP}_{d+3-i}\left( 
X_{i-1}\setminus \{ s_1,s_2,\ldots,s_{i-1} \}
\right)
\]
is even, where $X_{i-1}$ has color classes $C^{i-1}_0,C^{i-1}_1,\ldots, C^{i-1}_d$ such 
that $|C^{i-1}_j|=d+4-i$ for $j\geq d-i+2$, and $|C^{i-1}_j|=d+3-i$ otherwise.

Let $q_i$ be a point of $X_{i-1}$, and the boundary hyperplane -- that 
is crossed transversally -- be spanned by a subset $H_{i}$ of $X_{i-1}$ such that 
$G_i=H_i\cup\{q_i\}$ is rainbow. We distinguish two cases: 
\begin{enumerate}[{Case (i},1)]
\item\label{fall1} $G_i\cap C^{i-1}_{d-i+2}=\emptyset$.
\item\label{fall2} $G_i\cap C^{i-1}_{d-i+2}\not=\emptyset$.
\end{enumerate}

In Case~(i,\ref{fall1}), we show that a pairing for the colored Birch partitions shows up:
For every point $r\in C^{i-1}_{d-i+2}$, the property of containing the origin changes for the colored 
$d$-simplices $G_i\cup \{r\}$ while $q_i$ crosses the hyperplane through $H_i$, due to~(\ref{crossing}). 
A $d$-simplex $G_i\cup \{r\}$ contributes to the number of colored Birch partitions if the rest
admits a Birch partition into $d+2-i$ blocks. The latter property is independent of the current moving
process. In fact, $G_i\cup \{r\}$ contributes a summand
\[\text{cBP}_{d+3-i}\left( 
X_{i-1}\setminus \{ s_1,s_2,\ldots s_{i-2},s_{i-1} \}
\right),
\]
where $s_1\in C^{i-1}_d,s_2\in C^{i-1}_{d-1},\ldots,s_{i-1}\in C^{i-1}_{d-i+2}$, and 
$r\not=s_{i-1}$, in a positive, or negative way. This contribution can be concretized to be 
\[\text{cBP}_{d+2-i}\left( 
X_{i-1}\setminus (G_i\cup\{ s_1,s_2,\ldots,s_{i-1},r \})
\right).
\]

But the same 
contribution shows up for the colored $d$-simplex $G_i\cup \{s_{i-1}\}$ in the summand 
\[\text{cBP}_{d+3-i}\left( 
X_{i-1}\setminus \{ s_1,s_2,\ldots,s_{i-2},r \}
\right),
\]
again in a positive, or negative way. 
In any case, the parity of $\text{cBP}_{d+3-i}(X_{i-1})$ remains unchanged.

In Case~(i,\ref{fall2}), let $r$ be the unique point in $G_i\cap C^{i-1}_{d-i+2}$. Then
all summands
\[\text{cBP}_{d+3-i}\left( 
X_{i-1}\setminus \{ s_1,s_2,\ldots,s_{i-2},r \}
\right)
\] 
do not change, as any colored $d$-simplex not containing $G_i$ is not affected.

We fix a point $s\in C^{i-1}_{d-i+2}$, such that $s\not= r$. Assume
without restriction that in $H_i\cup\{q_i\}$ the $d$ colors $0,1,\ldots,\widehat{d-i+1},\ldots,d$ show up 
such that $C_{d-i+1}^{i-1}\cap G_i=\emptyset$. For every point $t\in C_{d-i+1}^{i-1}$, 
the property of containing the origin changes for the colored $d$-simplex
$G_i\cup \{t\}$, when $q_i$ crosses the hyperplane through $H_i$. 
Every simplex $G_i\cup \{t\}$ contributes
\[\text{cBP}_{d+2-i}\left( X_{i-1}\setminus (G_i\cup \{s,t\})\right)\]
to $\text{cBP}_{d+3-i}\left( X_{i-1}\setminus \{s\}\right)$ in a positive, or negative way. 
Note that the expression above is a sum.

Hence, it is sufficient for Case~(i,\ref{fall2}) to show that all these contributions sum up to an even 
number:
\[\text{cBP}_{d+2-i}(X_i)=
\sum_{s_1\in C^{i}_d,s_2\in C^{i}_{d-1},\ldots,s_{i}\in C^{i}_{d-i+1}}\
\text{cBP}_{d+2-i}\left( 
X_{i}\setminus \{ s_1,s_2,\ldots,s_{i} \}
\right),
\]
where $X_i=X_{i-1}\setminus G_i$. $X_i$ has color classes $C^i_j$, where $C^i_j$ is obtained from $C^{i-1}_j$
by deleting the point of color $j$ in $G_i$ for all $0\leq j\leq d$. Note that $|C^i_j|=d+3-i$,
for $j\geq d-i+1$; otherwise $|C^2_j|=d+2-i$.

Case~(i,\ref{fall2}) of Step $i$ reduces our original problem in the following way: The parameter
$k=d+3-i$ is reduced by $1$ to $k=d-2+i$, and the number of points is reduced by $d$.\\

{\bf After step $d$:} The outcome of this procedure is a colored set $X_d$ with the color class 
$C^d_0$ of size $2$, and color classes $C^d_1$ to $C^d_d$ of size $3$. It remains to
prove that
\[ \text{cBP}_2 (X_d) =
\sum_{s_1\in C^d_d,s_2\in C^d_{d-1},\ldots,s_d\in C^d_1}\
\text{cBP}_2\left( 
X_d\setminus \{ s_1,s_2,\ldots,s_d \}
\right)
\text{ is even.}\]

For this, let $q_{d+1}$ be a point of $X_d$, and the boundary hyperplane -- that 
is crossed transversally -- be spanned by a subset $H_{d+1}$ of $X_d$ such that 
$G_{d+1}=H_{d+1}\cup\{q_{d+1}\}$ is rainbow. We distinguish to cases 
\begin{enumerate}[{Case (d+1,}1)]
\item\label{final-fall1} $G_{d+1}\cap C^d_0=\emptyset$.
\item\label{final-fall2} $G_{d+1}\cap C^d_0\not=\emptyset$.
\end{enumerate}
In Case~(d+1,\ref{final-fall1}), a pairing shows up as in the previous steps. Analogously,
Case~(d+1,\ref{final-fall2}) reduces to the statement of Lemma~\ref{lem-key}.
\end{proof}

\begin{proof}[Proof of Property~(\ref{it-cb-even}) implies Properties~(\ref{it-cb-lower-plane}) and~(\ref{it-cb-lower})]
For $d\geq 2$, let us first prove 
\begin{eqnarray}
\text{cBP}_k(X)>0\,\Longrightarrow\,\text{cBP}_k(X)\geq 2^{k-d-1}\text{ for }d\geq 2 
\text{ and }k\geq d+2,
\end{eqnarray} 
via an induction on $k\geq d+2$. This settles Property~(\ref{it-cb-lower}).

Property~(\ref{it-cb-even}) implies the base case $k=d+2$: 
\[\text{cBP}_k(X)>0\,\,\Longrightarrow\,\,\text{cBP}_k(X)\geq 2=2^{k-d-1}.\]

Let now $k\geq d+3$, and be $\text{cBP}_k(X)>0$.
Then there is a colored Birch partition $B_1,B_2,\ldots,B_k$ of $X$. For $1\leq i\leq k$,
let $x_i$ be the point of color $0$ such that $x_i \in B_i$.
Note that
for any non-empty subset $I$ of the index set $[k]$, the set $\bigcup_{i\in I}B_i$
has again a colored Birch partition.

Using the base case for $\bigcup_{i\in [4]}B_i$, we obtain a
second colored Birch partition $B'_1,B'_2,B'_3,B'_4$ such that $x_i\in B'_i$ for all $i\in[4]$. 
Without loss of generality, we can assume $B_1\not=B'_1$.
Applying the assumption to the set $X\setminus B_1$, we obtain at least $2^{k-d-2}$ colored 
Birch partitions of $X$ starting with $B_1$. Finally, applying the assumption to the 
set $X\setminus B'_1$, we obtain again at least $2^{k-d-2}$ Birch partitions of $X$ starting
with $B'_1$. The construction of the sets $B_1$ and $B'_1$ leads to the factor of $2$.

To prove Property~(\ref{it-cb-lower-plane}), we show in the two subsequent paragraphs that a third set $B''_1$
can be constructed for $d=2$, and $k\geq 7$ so that all three sets a) contain a fixed point $x$, and b) are 
pairwise distinct. Therefore, the factor $3$ shows up in the lower bound for $d=2$ and $k\geq 7$.

For $x_1\in B_1$, the set $B'_1$ can be constructed as above. Now $B'_1$ contains a point $y\not= x_1$ 
that is not in $B_1$, and without loss of generality we can assume $y\in B_2$. Therefore $B_2\not=B'_2$.
The set $\{4,5,6,7\}$ has ${4 \choose 2}=6$ subsets $I$ with two elements. For every subset $I=\{ i_1,i_2\}$, 
we apply the base case to $B_1\cup B_3 \cup B_{i_1}\cup B_{i_2}$ so that we obtain each time a new colored 
Birch partition $B^I_1,B^I_3,B^I_{i_1},B^I_{i_2}$, such that $x_1\in B^I_1$, $x_3\in B^I_3$,
and $x_j\in B^I_j$ for both $j\in I$. If $B_1\not=B^I_1$ for one subset $I$, then $B'_1$ and $B^I_1$ 
are distinct by construction. Choosing $B''_1=B^I_1$ completes our proof.

If $B_1=B^I_1$ for all subsets $I$, then we proceed as follows:
For every $I$, there is a pair of $(i,j)$ from $I\cup\{3\}$ so that $B_i\not=F^I_i$ and $B_j\not=B^I_j$.
A pair of the form $(3,j)$ is the outcome of at most three index sets, and a pair of the form $(i,j)$ 
of at most two index sets, where $i,j\in \{4,5,6,7\}$. As we have in total $6$ pairs of indices, one index
$j\in \{3,4,5,6,7\}$ shows up in at least two pairs for two subsets $I_1,I_2$. 
Choosing the sets $B_j$, $B^{I_1}_j$, and $B^{I_2}_j$ completes our proof.
\end{proof}
\begin{rem}
It is easy to show that evenness - the key property of our proof - fails in the more general case of continuous maps.
If we count preimages instead of partitions, it should be possible to obtain a result similar to Theorem~\ref{thm-num-cbp}.
\end{rem}
\end{section}
\begin{section}{Proof of Theorem~\ref{thm-num-ctv}}\label{sec-proof-ctv}
\bprf The existence of at least one colored Tverberg partitions follows from
Theorem~\ref{thm-ctv-bmz11}. In the worst case, the partition consists of $d$ partition blocks of exactly
$d$ points each, intersecting in a single point, plus $r-d$ many $(d+1)$-element sets 
$B_1, B_2,\ldots , B_{r-d}$ that all contain the intersection point in their convex hulls; here we need 
$r \geq d$. The first $d+1$ colors show up exactly $r-d$ times if the unique point of color $d+1$ does 
not end up in one of the $B_i$'s. In that case, this single point is recolored with the unique color showing 
up $r-d-1$ times. In both cases, the each property of $T(f)$ follows directly from the corresponding property
for colored Birch partitions for $k=r-d$.
\eprf
\begin{rem}
\begin{enumerate}
\item In Theorem~\ref{thm-num-ctv}, the assumption $r$ prime is needed for the existence of at least one partition. 
Alternatively, $r$ arbitrary and $T(f)>0$ are sufficient conditions. 
\item Any lower bound $\ell$ on the number of colored Tverberg points for a given map $f$ 
improves our lower bounds for the number of colored Tverberg partitions by the factor of $\ell$.
\end{enumerate}
\end{rem}
\end{section}
\begin{section}{Remarks}\label{sec-rem}
Let us conclude this paper with a discussion on lower bounds for the
number of colored Tverberg partitions in the setting of Theorem~\ref{thm-ctv-bmz11}. The table
below shows minimal numbers based on four different approaches for $d=2$, and $r$ up to $8$.

The point configuration due to Sierksma~\cite[Sect.~6.6]{matou03:_using_borsuk_ulam} given by
$r-1$ points clustered around each of the vertices of a standard $d$-simplex in $\R^d$ plus one point
in its center seemed to be a good candidate for a minimal configuration, even in the colored case.
The Sierksma configuration is extremal in two ways: It has one
only Tverberg point, but this Tverberg point comes with the maximal number of Birch partitions.

The coloring of the Sierksma configuration which seems to lead to the smallest number of colored
Tverberg partitions is obtained as follows:
The point in the center is of color $d+1$. The $r-1$ points of every vertex are colored so 
that each of the colors $0,1,\ldots,d $ shows up  $(r-1)/(d+1)$ times, or 
$\lceil (r-1)/(d+1)\rceil$ resp.~$\lfloor (r-1)/(d+1)\rfloor$ if $r$ is not a multiple of $d+1$. 
In that case, the remaining $(r-1)$ modulo $(d+1)$ points of vertex $i$, where $0\leq i\leq d$, 
are colored in a cyclic way with colors $i, i+1, \ldots $ (modulo $d$). 
The number of colored Tverberg partitions for this colored configuration can be calculated 
for every dimension $d$ via recursion formulas. These numbers are shown in the table below.

Linda Kleist~\cite{note_linda11} who wrote her bachelor thesis under the supervision of Ziegler studied colored 
point configurations for $r\leq 6$, and $d=2$: The vertices of a regular 3(r-1)-gon plus its center point. 
Minimizing over all colorings, this construction led to larger numbers. 
Her results are shown below.

While experimenting with randomly placed colored points in the plane, we obtained minimal numbers
shown in the table below. Looking at $100000$ examples for $r=5$ has led to five colored sets 
with 10 colored Tverberg partitions. These minimal examples have several Tverberg points: 
One of the points of $X$, and intersection points of two segments. These examples kept us from coming up
with a conjecture for the number of colored Tverberg partitions based on the Sierksma configuration,
and the coloring from above.

The last column of our table shows the lower bound of Theorem~\ref{thm-num-ctv}. The lower bound is shown
in brackets for $r$ non-prime. In that case, the additional assumption $T(f)>0$ is needed. 
\begin{center}
\begin{tabular}{|c|c|c|c|c|}
\hline
  \multirow{3}{*}{$r$} & Minimum for  & Minimum for polygonal & Minimum for  & Lower bound\\ 
    & colored Sierksma  & configurations  & random  		& of \\
	& configurations	& due to Kleist	  & configurations	&  Theorem~\ref{thm-num-ctv}\\ \hline
  	2 & 1 & 1 & 1 & 1\\\hline
	3 & 1 & 1 & 1 & 1\\\hline
	4 & 2 & 2 & 2 &(1)\\\hline
	5 & 12 & 16 & 10 & 1\\\hline
	6 & 80 & 80 & 80 & (2)\\\hline
	7 & 640 & - & 864 & 4\\\hline
	8 & 9216 & - &  $>10000$& (8)\\\hline
\end{tabular}
\end{center}

In conclusion, the table suggests 1) that finding minimal colored configurations is not easy, 2) that
looking at random configurations fails for $r>6$, and 3) that our lower bound is not tight.\\

{\bf Acknowledgements.} The author is grateful to Pavle Blagojevi\'c for critical comments,
and to G\"unter M.~Ziegler for critical remarks which led to an improvement of the paper.
\end{section}


\begin{thebibliography}{10}

\bibitem{barany92}
\textsc{I.~B\'ar\'any and D.~G. Larman}, \emph{A colored version of
  {T}verberg's theorem}, J.~London Math.~Soc.~(2) \textbf{45} (1992),
  pp.~314--320.

\bibitem{bmz10:_col-tp}
\textsc{P.~V.~M. Blagojevi\'c, B.~Matschke, and G.~M. Ziegler}, \emph{Optimal
  bounds for the colorful {T}verberg problem}.
\newblock Preprint, October 2009, 10~pages; revised November 2009, 11~pages;
  \url{http://arXiv.org/abs/0910.4987}.

\bibitem{bmz11:_col-tvp}
\textsc{P.~V.~M. Blagojevi\'c, B.~Matschke, and G.~M. Ziegler}, \emph{Optimal
  bounds for a colorful {T}verberg-{V}re\'cica type problem}, Adv.~Math.
  \textbf{226} (2011), pp.~5198--5215.

\bibitem{deza06:_colour}
\textsc{A.~Deza, S.~Huang, T.~Stephan, and T.~Terlaky}, \emph{Colourful
  simplicial depth}, Discrete Comp.~Geom. \textbf{35} (2006), pp.~597--615.

\bibitem{hell07:_tverb}
\textsc{S.~Hell}, \emph{On the number of {T}verberg partitions in the prime
  power case}, Europ.~J.~of Comb. \textbf{28} (2007), pp.~347--355.

\bibitem{hell08:_birch}
\textsc{S.~Hell}, \emph{On the number of {B}irch partitions}, Discrete
  Comput.~Geom. \textbf{40} (2008), pp.~586--594.

\bibitem{note_linda11}
\textsc{L.~Kleist}, \emph{Zehn bunte {P}unkte in der {E}bene}, Mitteilungen der
  DMV \textbf{19} (2011), p.~124.
\newblock Contribution to a puzzle of Matschke, and Ziegler in MDMV {\bf 18}
  (2010).

\bibitem{matou03:_using_borsuk_ulam}
\textsc{J.~Matou{\v s}ek}, \emph{Using the {B}orsuk--{U}lam theorem},
  Universitext, Springer--Verlag, Berlin, 2003.
\newblock Lectures on topological methods in combinatorics and geometry.

\bibitem{mtw11:_col-geom-ctvp}
\textsc{J.~Matou\v{s}ek, M.~Tancer, and U.~Wagner}, \emph{A geometric proof of
  the colored {T}verberg theorem}, Discrete Comput.~Geom. \textbf{47}, no.~2
  (2012), pp.~245--265.

\bibitem{tverberg66:_theor}
\textsc{H.~Tverberg}, \emph{A generalization of {R}adon's theorem}, J.~London
  Math.~Soc \textbf{41} (1966), pp.~123--128.

\bibitem{vrecica-zival-11:chessboard}
\textsc{S.~T. Vre\'{c}ica and R.~T. \v{Z}ivaljevi\'c}, \emph{Chessboard
  complexes indomitable}, J.~Comb.~Theory, Ser.~A \textbf{118} (2011),
  pp.~2157--2166.

\bibitem{vz93:_notes_sierk}
\textsc{A.~Vu\v{c}i\'c and R.~T. \v{Z}ivaljevi\'c}, \emph{Notes on a conjecture
  of {S}ierksma}, Discrete Comput.~Geom. \textbf{9} (1993), pp.~339--349.

\bibitem{z11:_col-tvp_ams}
\textsc{G.~M. Ziegler}, \emph{3{N} colored points in a plane}, Notices of the
  AMS \textbf{58} (2011), pp.~550--557.

\end{thebibliography}

Email-address: \href{mailto:stephan@hell-wie-dunkel.de}{stephan@hell-wie-dunkel.de}

\end{document}